\pdfoutput=1

\documentclass[reqno,a4paper,final]{amsart}

\usepackage[utf8]{inputenc}
\usepackage[T1]{fontenc}
\usepackage[english]{babel}
\usepackage{ifthen}
\usepackage{enumitem}
\usepackage{dsfont}
\usepackage{mathtools}
\usepackage[babel]{csquotes}
\usepackage[backend=biber,style=alphabetic]{biblatex}
\usepackage[protrusion=true,expansion=true,babel=true,final]{microtype}
\usepackage[unicode,bookmarksopen]{hyperref}

\numberwithin{equation}{section} 

\newtheorem{lemma}{Lemma}[section]
\newtheorem{corollary}[lemma]{Corollary}
\newtheorem{proposition}[lemma]{Proposition}
\newtheorem{theorem}[lemma]{Theorem}

\theoremstyle{definition}
\newtheorem{definition}[lemma]{Definition}
\newtheorem{example}[lemma]{Example}
\newtheorem{remark}[lemma]{Remark}

\newlist{thm_enum}{enumerate}{1}
\setlist[thm_enum]{label=\normalfont(\alph*)}
\newlist{def_enum}{enumerate}{1}
\setlist[def_enum]{label=\normalfont(\roman*)}
\newlist{equiv_enum}{enumerate}{1}
\setlist[equiv_enum]{label=\normalfont(\roman*)}

\newcommand{\IN}{\mathbb{N}}
\newcommand{\IR}{\mathbb{R}}
\newcommand{\IC}{\mathbb{C}}
\newcommand{\IE}{\mathbb{E}}

\newcommand{\abs}[1]{\left\lvert#1\right\rvert}
\newcommand{\normalabs}[1]{\lvert#1\rvert}

\newcommand{\norm}[1]{\left\lVert#1\right\rVert}
\newcommand{\normalnorm}[1]{\lVert#1\rVert}
\newcommand{\biggnorm}[1]{\biggl\lVert#1\biggr\rVert}

\newcommand{\R}[2][\empty]{
	\ifthenelse{\equal{#1}{\empty}}
		{\mathcal{R}\left\{#2\right\}}
		{\mathcal{R}_{#1}\left\{#2\right\}}
}

\makeatletter
\newcommand{\LeftEqNo}{\let\veqno\@@leqno}
\makeatother

\renewcommand{\d}{\mathop{}\!d}
\renewcommand{\Re}{\operatorname{Re}}
\renewcommand{\Im}{\operatorname{Im}}
\renewcommand{\epsilon}{\varepsilon}
\renewcommand{\phi}{\varphi}

\DeclareMathOperator{\E}{\IE}

\DeclareMathOperator{\sign}{sign}

\DeclareMathOperator{\Rad}{Rad}

\allowdisplaybreaks[4]

\renewbibmacro{publisher+location+date}{%
	\printlist{publisher}
	\setunit*{\addcomma\space}
	\printlist{location}
  	\setunit*{\addcomma\space}
  	\usebibmacro{date}
\newunit} 

\newbibmacro{string+doi}[1]{%
	\iffieldundef{doi}{#1}{\href{http://dx.doi.org/\thefield{doi}}{#1}}}

\renewbibmacro{in:}{}

\DeclareFieldFormat*{title}{\usebibmacro{string+doi}{\mkbibemph{#1}}}
\DeclareFieldFormat*{booktitle}{#1}
\DeclareFieldFormat[article]{volume}{\textbf{#1}\addcomma}
\DeclareFieldFormat[article]{number}{\addspace no.~#1}
\DeclareFieldFormat[article]{pages}{#1}
\DeclareFieldFormat[book,incollection]{volume}{}
\DeclareFieldFormat[book,incollection]{series}{#1\addcomma\space vol.~\thefield{volume}}
\DeclareFieldFormat{journaltitle}{#1} 

\renewcommand*{\bibfont}{\footnotesize}

\begin{document}

\title[Counterexamples to Maximal Regularity]{Maximal Regularity: Positive Counterexamples on UMD-Banach Lattices and Exact Intervals for the Negative Solution of the Extrapolation Problem}

\begin{abstract}
	Using methods from Banach space theory, we prove two new structural results on maximal regularity. The first says that there exist positive analytic semigroups on UMD-Banach lattices, namely $\ell_p(\ell_q)$ for $p \neq q \in (1, \infty)$, without maximal regularity. In the second result we show that the extrapolation problem for maximal regularity behaves in the worst possible way: for every interval $I \subset (1, \infty)$ with $2 \in I$ there exists a family of consistent bounded analytic semigroups $(T_p(z))_{z \in \Sigma_{\pi/2}}$ on $L_p(\IR)$ such that $(T_p(z))$ has maximal regularity if and only if $p \in I$.
\end{abstract}

\author{Stephan Fackler}
\address{Institute of Applied Analysis, University of Ulm, Helmholtzstr. 18, 89069 Ulm}
\email{stephan.fackler@uni-ulm.de}
\thanks{The author was supported by a scholarship of the ``Landesgraduiertenförderung Baden-Württemberg''.}

\maketitle

\section{Introduction}

		Let $-A$ be the generator of a $C_0$-semigroup $(T(t))_{t \ge 0}$ on a Banach space $X$. One says that $-A$ has \emph{maximal regularity} (for one or equivalently all choices of $T>0$ and of $p \in (1, \infty)$) if for all $f \in L_p([0,T];X)$ the mild solution $u(t) = \int_0^t T(t-s) f(s) \d s$ of the inhomogeneous abstract Cauchy problem
			\[ 
				\left\{ \begin{aligned} 
					\dot{u}(t) + A(u(t)) & = f(t) \\
					u(0) & = 0
				\end{aligned} \right.
			\]
		satisfies $u \in W_{p}^{1}([0,T];X) \cap L_p([0,T]; D(A))$. Maximal regularity is a fundamental tool in the study of non-linear partial differential equations (see \cite{KunWei04}, \cite{DHP03}, \cite{Pru02} and the references therein).
		
		Although maximal regularity has been very successful for concrete applications, fundamental questions in the structural understanding of this concept are still open (for an explicit statement see for example~\cite[Section~7]{Kal01}). For example, until recently, no explicit example of a generator of an analytic semigroup without maximal regularity on $L_p$ for $p \in (1, \infty) \setminus \{2\}$ has been known, although the existence was shown in \cite{KalLan00}.  In fact, the first explicit examples were given by the author in \cite{Fac13} and \cite{Fac14}. The aim of this note is to develop the techniques from \cite{Fac14} further in order to give new contributions to the study of the structure of maximal regularity with the help of the theory of Schauder bases. For a survey on the open questions in this area we refer to \cite{Fac14b}.
		
		In the first main result (Theorem~\ref{thm:positive_counterexamples}) we show that there exist generators of positive analytic semigroups on UMD-Banach lattices, namely on $\ell_p(\ell_q)$ for $p \neq q \in (1, \infty)$, without maximal regularity. Hence, positivity does not imply maximal regularity. This is in contrast to the following positive result due to L.~Weis \cite[Remark~4.9c)]{Wei01}: the generator of a bounded analytic semigroup on some $L_p$-space for $p \in (1, \infty)$ that is positive and contractive on the real line has maximal regularity.
		
		In the second main result (Theorem~\ref{thm:extrapolation_mr_complete_counterexample}) we study the extrapolation problem for maximal regularity. In \cite{Fac14} it was shown that maximal regularity does in general not extrapolate from $L_2$ to the $L_p$-scale, i.e.\ there exists a family of consistent semigroups $(T_p(t))_{t \ge 0}$ on $L_p$ for $p \in (1, \infty)$ such that $(T_2(t))_{t \ge 0}$ has maximal regularity, but $(T_p(t))_{t \ge 0}$ fails maximal regularity for some, indeed all, $p \in (1, \infty) \setminus \{2\}$. Here we extend this negative result. Indeed, we show that the extrapolation problem behaves in the worst possible way. We now shortly explain what this means.
		
	Suppose that one has given a family $(T_p(z))$ of consistent analytic $C_0$-semigroups on $L_p$ for $p \in (1, \infty)$ and let $M \subset (1,\infty)$ be the set all $p \in (1, \infty)$ for which the semigroup $(T_p(z))$ has maximal regularity. Since an analytic $C_0$-semigroup on a Hilbert space has maximal regularity by a result of de Simon \cite[Lemma 3,1]{Sim64}, one has $2 \in M$. Moreover, it follows from complex interpolation that $M$ is a subinterval of $(1, \infty)$. We show that apart from these obvious structural restrictions one cannot obtain any further positive results for the extrapolation problem: for every interval $I \subset (1, \infty)$ with $2 \in I$ there exists a family of consistent $C_0$-semigroups $(T_p(z))_{z \in \Sigma_{\frac{\pi}{2}}}$ on $L_p(\IR)$ such that $(T_p(z))_{z \in \Sigma_{\frac{\pi}{2}}}$ has maximal regularity if and only if $p \in I$.
	
	In contrast, positive results for the extrapolation problem are known under additional assumptions on the semigroups (see~\cite[Section~6]{Fac13+} and the references therein).

\section{\texorpdfstring{$\mathcal{R}$}{R}-Sectorial Operators and Associated Operators}

	In this section we present the necessary background on sectorial operators and maximal regularity.

	\begin{definition}[Sectorial Operator]\index{sectorial operator} A densely defined operator $A$ on a Banach space $X$ is called \emph{sectorial} if there exists an $\omega \in (0, \pi)$ such that
		\begin{equation*}
			\label{sectorial}
			\tag{$S_{\omega}$}
			\sigma(A) \subset \overline{\Sigma_{\omega}} \qquad \text{and} \qquad \sup_{\lambda \not\in \overline{\Sigma_{\omega + \epsilon}}} \norm{\lambda R(\lambda, A)} < \infty \quad \forall \epsilon > 0.
		\end{equation*}
		One defines the \emph{sectorial angle of $A$}\index{sectorial operator!angle}\index{$\omega(A)$|see{sectorial angle}}\index{angle!sectorial} as $\omega(A) \coloneqq \inf \{ \omega: \text{\eqref{sectorial} holds} \}$.		
	\end{definition}
	
	Recall that on Banach spaces sectorial operators are exactly the (negative) generators of strongly continuous analytic semigroups. Maximal regularity can be characterized by a stronger boundedness property than the boundedness in operator norm. Let $r_k(t) \coloneqq \sign \sin (2^k \pi t)$ be the $k$-th \emph{Rademacher function}.
	
	\begin{definition}[$\mathcal{R}$-Boundedness]
		A family of operators $\mathcal{T} \subset \mathcal{B}(X)$ on a Banach space $X$ is called \emph{$\mathcal{R}$-bounded}\index{$\mathcal{R}$-bounded} if there exists a finite constant $C \ge 0$ such that for each finite subset $\{T_1, \ldots, T_n \}$ of $\mathcal{T}$ and arbitrary $x_1, \ldots, x_n \in X$ one has
		\begin{equation} 
			\biggnorm{\sum_{k = 1}^n r_k T_k x_k}_{L_2([0,1]; X)} \le C \biggnorm{\sum_{k=1}^n r_k x_k}_{L_2([0,1]; X)}. \label{eq:R-ineq}
		\end{equation}
		The best constant $C$ such that \eqref{eq:R-ineq} holds is denoted by $\mathcal{R}(\mathcal{T})$.
	\end{definition} 

	Further let $\Rad(X)$ denote the closed linear span in $L_2([0,1];X)$ of functions of the form $\sum_{k=1}^n r_k x_k$ for $n \in \IN$ and $x_1, \ldots, x_n \in X$. We need the following strengthening of the sectoriality condition.
		
	\begin{definition}[$\mathcal{R}$-Sectorial Operator]			
		A sectorial operator $A$ on a Banach space $X$ is called \emph{$\mathcal{R}$-sectorial} if for some $\omega > \omega(A)$ one has
		\begin{equation*}
			\mathcal{R} \{ \lambda R(\lambda,A): \lambda \not\in \overline{\Sigma_{\omega}} \} < \infty. \label{R-sectorial}\tag{$\mathcal{R}_{\omega}$}
		\end{equation*}
		One defines the \emph{$\mathcal{R}$-sectorial angle} as $\omega_R(A) \coloneqq \inf\{ \omega: \text{\eqref{R-sectorial} holds} \}$\index{angle!$\mathcal{R}$-sectorial}\index{$\mathcal{R}$-sectorial operator!angle}\index{$\omega_R(A)$|see{$\mathcal{R}$-sectorial angle}}.
		\end{definition}
	
	One has the following connection between maximal regularity and $\mathcal{R}$-sectorial operators by L.~Weis \cite[Theorem~4.2]{Wei01} (the UMD assumption is actually only needed in one implication). For the definition and further information on UMD-spaces we refer to \cite{Fra86} and \cite{Bur01}. In the following we only need that mixed $L_p(L_q)$-spaces for $p, q \in (1, \infty)$ are UMD. 
	
	\begin{theorem}\label{thm:weis} Let $-A$ be the generator of an analytic $C_0$-semigroup $(T(z))_{z \in \Sigma}$ on a Banach space $X$. Then the following hold:
		\begin{thm_enum}
			\item If $-A$ has maximal regularity, then $A$ is $\mathcal{R}$-sectorial with $\omega_R(A) < \frac{\pi}{2}$.
			\item Conversely, if $X$ is UMD, $-A$ has maximal regularity if $A$ is an $\mathcal{R}$-sectorial operator with $\omega_R(A) < \frac{\pi}{2}$.
		\end{thm_enum}
	\end{theorem}
	
	Hence, it suffices to construct sectorial operators which are not $\mathcal{R}$-sectorial to give counterexamples to maximal regularity. We now transfer the $\mathcal{R}$-sectoriality of an operator on $X$ to the boundedness of an associated operator on $\Rad(X)$. This variant of the transference result in \cite[Theorem~3.3]{Fac14} (see also \cite[Proposition~3.16]{Fac14b} and \cite[Theorem~3.6]{AreBu03}) is the central tool for the counterexamples to be given later.

	\begin{proposition}\label{prop:resolvent_r_associated_operator}
		Let $A$ be an $\mathcal{R}$-sectorial operator. Then there exists a constant $C \ge 0$ such that for all $(q_n)_{n \in \IN} \subset \IR_{-}$ the associated operator
			\[ \mathcal{R}\colon \sum_{n=1}^{N} r_n x_n \mapsto \sum_{n=1}^{N} r_n q_n R(q_n,A)x_n \]
		defined on the finite Rademacher sums extends to a bounded operator on $\Rad(X)$ with operator norm at most $C$.
	\end{proposition}
	\begin{proof}
		Since $A$ is $\mathcal{R}$-sectorial, one has $C \coloneqq \mathcal{R}\{ \lambda R(\lambda, A): \lambda \in \IR_{-} \} < \infty$. Hence, for all finite Rademacher sums we have by the definition of $\mathcal{R}$-boundedness
			\[ \biggnorm{\sum_{n=1}^N r_n q_n R(q_n,A)x_n} \le C \biggnorm{\sum_{n=1}^N r_n x_n}. \qedhere \]
	\end{proof}

\section{Positive Analytic Semigroups without Maximal Regularity}
			
	In this section we construct generators of positive analytic semigroups without maximal regularity. Let $X$ be a Banach space that admits an $1$-unconditional Schauder basis $(e_m)_{m \in \IN}$. Then it is well-known that $(e_m)_{m \in \IN}$ induces on $X$ via
		\[ x = \sum_{m=1}^{\infty} a_m e_m \ge 0 \quad :\Leftrightarrow \quad a_m \ge 0 \quad \text{for all } m \in \IN \]
	the structure of a Banach lattice. Let $\pi\colon \IN \to \IN$ be a permutation of the even numbers. Then by a classical perturbation result $(f_m)_{m \in \IN}$ defined by
		\begin{equation}
			\label{eq:basis_perturbation}
			f_m =
			\begin{cases}
				e_m & m \text{ odd} \\
				e_{m-1} + e_{\pi(m)} & m \text{ even}
			\end{cases}
		\end{equation}
	is a Schauder basis for $X$ \cite[Ch.~I, Proposition~4.4]{Sin70}. Let $A$ be the \emph{Schauder multiplier} associated to some real positive sequence $(\gamma_m)_{m \in \IN}$ with respect to $(f_m)_{m \in \IN}$, that is 
		\begin{align*}
			D(A) = \biggl\{ x = \sum_{m=1}^{\infty} a_m f_m: \sum_{m=1}^{\infty} \gamma_m a_m f_m \text{ exists} \biggr\} \\
			A \left(\sum_{m=1}^{\infty} a_m f_m \right) = \sum_{m=1}^{\infty} \gamma_m a_m f_m.
		\end{align*}
		Then $A$ clearly is a closed densely defined operator. Let $BV$ be the Banach space of all sequences $(x_m)_{m \in \IN}$ with bounded variation, i.e.\ $\norm{(x_m)}_{\mathrm{BV}} \coloneqq \norm{x_1} + \sum_{m=1}^{\infty} \abs{x_{m+1} - x_m} < \infty$. Concerning the boundedness of $A$, one has the following positive result \cite[Lemma~2.4]{Ven93}.
		
		\begin{proposition}\label{prop:sm_bounded}
			Let $(e_m)_{m \in \IN}$ be a Schauder basis for a Banach space $X$. Then there exists a constant $K \ge 0$ such that for all $(\gamma_m)_{m \in \IN} \in BV$ the Schauder multiplier $A$ associated to $(\gamma_m)_{m \in \IN}$ with respect to $(e_m)_{m \in \IN}$ is bounded and satisfies
				\[
					\norm{A} \le K \norm{(\gamma_m)}_{\mathrm{BV}}.
				\]
		\end{proposition}
		 
	We now study the positivity of the formal semigroup $(e^{-tA})_{t \ge 0}$ on $X$ generated by $-A$ as defined above with respect to the just defined lattice structure. Clearly, the semigroup is positive if and only if $e^{-tA} e_m \ge 0$ for all $m \in \IN$ and all $t \ge 0$. For odd $m$ this is satisfied because of $e^{-tA}e_m = e^{-tA} f_m = e^{-\gamma_m t} e_m$. For even $m$ one has
		\begin{equation}\label{eq:application_semigroup}
			\begin{split} 
				e^{-tA} e_m & = e^{-tA} (f_{\pi^{-1}(m)} - e_{\pi^{-1}(m)-1}) = e^{-tA} (f_{\pi^{-1}(m)} - f_{\pi^{-1}(m)-1}) \\
				& = e^{-t\gamma_{\pi^{-1}(m)}} f_{\pi^{-1}(m)} - e^{-t \gamma_{\pi^{-1}(m)-1}} f_{\pi^{-1}(m)-1} \\
				& = e^{-t\gamma_{\pi^{-1}(m)}} (e_{\pi^{-1}(m)-1} + e_m) - e^{-t \gamma_{\pi^{-1}(m)-1}} e_{\pi^{-1}(m)-1} \\
				& = (e^{-t\gamma_{\pi^{-1}(m)}} - e^{-t \gamma_{\pi^{-1}(m)-1}}) e_{\pi^{-1}(m)-1} + e^{-t\gamma_{\pi^{-1}(m)}} e_m.
			\end{split}
		\end{equation}
	Therefore $(e^{-tA})_{t \ge 0}$ is positive if and only if $\gamma_{m} \le \gamma_{m-1}$ for all even $m \in \IN$. 
	
	Later, the following elementary observation will be useful.
	
	\begin{lemma}\label{lem:maximize_resolvent_sequence}
		For $\gamma_m > \gamma_{m-1} > 0$ consider the function $d(t) \coloneqq t [(t+\gamma_{m-1})^{-1} - (t+\gamma_{m})^{-1}]$ on $\IR_{+}$. Then $d$ has a maximum which is bigger than $\frac{1}{2} \frac{\gamma_m - \gamma_{m-1}}{\gamma_m + \gamma_{m-1}}$.
	\end{lemma} 
	
	With these preliminary observations we obtain the following result.
	
	\begin{theorem}\label{thm:positive_counterexamples}
		Let $X$ be a Banach space that admits a normalized non-symmetric 1-unconditional Schauder basis $(e_m)_{m \in \IN}$. We consider $X$ as a Banach lattice with the order induced by $(e_m)_{m \in \IN}$. Then there exists a non-$\mathcal{R}$-sectorial operator $A$ with $\omega(A) = 0$ such that $-A$ generates a positive analytic $C_0$-semigroup on $X$.
	\end{theorem}
	\begin{proof}
		It follows from \cite[first part of the proof of Proposition~23.2]{Sin70} that there exists a permutation $\pi\colon \IN \to \IN$ of the even numbers such that $(e_{2m - 1})_{m \in \IN}$ and $(e_{\pi(2m)})_{m \in \IN}$ are not equivalent. Hence, there exists a sequence $(a_m)_{m \in \IN}$ such that the expansion for $(a_m)_{m \in \IN}$ converges with respect to $(e_{2m - 1})_{m \in \IN}$ or $(e_{\pi(2m)})_{m \in \IN}$ but not for both. For the rest of the proof we will assume that the expansion converges for $(e_{\pi(2m)})_{m \in \IN}$. In the other case a similar argument can be applied if one replaces $(f_m)_{m \in \IN}$ by
			\[
				f_m = \begin{cases}
							e_m + e_{\pi(m+1)} & m \text{ odd} \\
							e_{\pi(m)} & m \text{ even}.
				\end{cases}
			\]
		We use the following twisted version of the lacunary sequence $(2^m)_{m \in \IN}$:
			\[
				\gamma_m = \begin{cases}
					2^{m+1} & m \text{ odd} \\
					2^{m-1} & m \text{ even}.
				\end{cases}
			\]
		By definition, one has $\gamma_m < \gamma_{m-1}$ for all even $m \in \IN$. By the above observation this implies that the formal semigroup generated by the Schauder multiplier associated to $(-\gamma_m)_{m \in \IN}$ is positive. It therefore remains to show that the multiplier $A$ associated to $(\gamma_m)_{m \in \IN}$ is a sectorial operator with $\omega(A) = 0$ which is not $\mathcal{R}$-sectorial. For this let us consider the sequence $(e^{-t\gamma_m^{\alpha}})_{m \in \IN}$ for $t > 0$ und $\alpha > 0$. For its total variation one obtains
			\begin{align*}
				\MoveEqLeft \sum_{m = 1}^{\infty} e^{-t 2^{(2m-1)\alpha}} - e^{-t 2^{2m\alpha}} + e^{-t 2^{(2m-1)\alpha}} - e^{-t 2^{2(m+1) \alpha}} \\
				& \le t \sum_{m=1}^{\infty} (2^{2m\alpha} - 2^{(2m-1)\alpha}) e^{-t2^{(2m-1)\alpha}} + (2^{(2m+2)\alpha} - 2^{(2m-1)\alpha}) e^{-t 2^{(2m-1)\alpha}} \\
				& = (2^{3\alpha} + 2^{\alpha} - 2)t \sum_{m=1}^{\infty} 2^{(2m-1)\alpha} e^{-t 2^{(2m-1)\alpha}} \\
				& = \frac{2^{3\alpha} + 2^{\alpha} - 2}{2^{\alpha} - 1}t \sum_{m=1}^{\infty} \int_{2^{(2m-1)\alpha}}^{2^{2m\alpha}} e^{-t2^{(2m-1)\alpha}} \, \d s \\
				& \le \frac{2^{3\alpha} + 2^{\alpha} - 2}{2^{\alpha} - 1}t \sum_{m=1}^{\infty} \int_{2^{(2m-1)\alpha}}^{2^{2m\alpha}} e^{-ts/2^{\alpha}} \, \d s \\
				& \le \frac{2^{3\alpha} + 2^{\alpha} - 2}{2^{\alpha} - 1}t \int_{2^{\alpha}}^{\infty} e^{-ts/2^{\alpha}} \, \d s = \frac{2^{\alpha}}{2^{\alpha} - 1} (2^{3\alpha} + 2^{\alpha} - 2) e^{-t}. 				
			\end{align*}
		It follows from $\omega(A^{\alpha}) = \alpha \omega(A)$ \cite[Theorem~15.16]{KunWei04} and Proposition~\ref{prop:sm_bounded} that $A$ is sectorial with $\omega(A) = 0$.
		
		Now assume that $A$ is $\mathcal{R}$-sectorial. Let $(q_m)_{m \in \IN} \subset \IR_{-}$ be a sequence to be chosen later. Then it follows from Proposition~\ref{prop:resolvent_r_associated_operator} that the operator $\mathcal{R}\colon \Rad(X) \to \Rad(X)$ associated to the sequence $(q_m)_{m \in \IN}$ is bounded. We now apply $\mathcal{R}$ to the element $x = \sum_{m=1}^{\infty} a_m e_{\pi(2m)} r_m$ of $\Rad(X)$. Since $R(\lambda, A)$ is the multiplier associated to the sequence $((\lambda - \gamma_m)^{-1})_{m \in \IN}$ and $e_{\pi(2m)} = f_{2m} - f_{2m-1}$, we obtain
			\begin{align*}
				\mathcal{R}(x) & = \mathcal{R} \biggl( \sum_{m=1}^{\infty} a_m (f_{2m} - f_{2m-1}) r_m \biggr) \\
				& = \sum_{m=1}^{\infty} r_m \frac{a_m q_m}{q_m - \gamma_{2m}} f_{2m} - r_m \frac{a_m q_m}{q_m - \gamma_{2m-1}} f_{2m-1} \\
				& = \sum_{m=1}^{\infty}r_m  \frac{a_m q_m}{q_m - \gamma_{2m}} (e_{\pi(2m)} + e_{2m-1}) - r_m \frac{a_m q_m}{q_m - \gamma_{2m-1}} e_{2m-1} \\
				& = \sum_{m=1}^{\infty} r_m \frac{a_m q_m}{q_m - \gamma_{2m}} e_{\pi(2m)} + r_m a_m q_m \biggl( \frac{1}{q_m - \gamma_{2m}} - \frac{1}{q_m - \gamma_{2m-1}} \biggr) e_{2m-1}.
			\end{align*}
		 Now take $q_m = -2^{2m-1}$ as motivated in Lemma~\ref{lem:maximize_resolvent_sequence}. Then we see that
			\[
				\sum_{m=1}^{\infty} \frac{1}{2} r_m a_m e_{\pi(2m)} + \frac{1}{6} r_m a_m e_{2m-1}.
			\]
		exists in $\Rad(X)$. By \cite[Theorem~12.3]{DJT95} and the unconditionality of the basis, $\sum_{m=1}^{\infty} a_m e_{2m-1}$ converges. This contradicts the choice of $(a_m)_{m \in \IN}$ and therefore $A$ is not $\mathcal{R}$-sectorial.
	\end{proof}
	
	We now give some concrete examples of spaces for which the above theorem can be applied.
	
	\begin{example}
		For $p, q \in (1, \infty)$ consider the UMD-spaces $\ell_p(\ell_q)$ with their natural lattice structure. Its ordering is induced by the standard unit vector basis $(e_m)_{m \in \IN}$ of $\ell_p(\ell_q)$ for some enumeration of $\IN \times \IN$. Clearly, $\ell_p(\ell_q)$ contains both copies of $\ell_p$ and $\ell_q$ and therefore for $p \neq q$ the basis $(e_m)_{m \in \IN}$ is 1-unconditional and non-symmetric. Hence for $p \neq q$, Theorem~\ref{thm:positive_counterexamples} yields a sectorial operator $A$ on $\ell_p(\ell_q)$ with $\omega(A) = 0$ such that $-A$ generates a positive analytic $C_0$-semigroup without maximal regularity.
	\end{example}
	
	\begin{example}
		In the next section we see that for $p \in (1,\infty) \setminus \{2\}$ the space $\ell_p$ admits after equivalent renorming a non-symmetric $1$-unconditional basis. If one uses the ordering induced by this basis, one sees with the help of Theorem~\ref{thm:positive_counterexamples} that one can give $\ell_p$ after equivalent renorming a non-standard lattice structure for which there exists a generator $-A$ of a positive analytic $C_0$-semigroup without maximal regularity satisfying $\omega(A) = 0$. Further, these arguments apply to every normalized unconditional basis of $L_p([0,1])$ for $p \in (1, \infty) \setminus \{2\}$ as such bases are automatically non-symmetric~\cite[Ch.~II,~Theorem~21.1]{Sin70}.
	\end{example}

\section{Exact Control of the Extrapolation Scale}
				
		In this section we give the announced complete negative solution of the extrapolation problem for maximal regularity. For $p \in (1, \infty)$ let $(e_m)_{m \in \IN}$ be the standard unit vector basis of $X_p \coloneqq (\oplus_{n=1}^{\infty} \ell_2^n)_{\ell_p}$ seen as a sequence space, the $\ell_p$-sum of finite dimensional Euclidean spaces of increasing dimension. Consider the basis $(f_m)_{m \in \IN}$ given by \eqref{eq:basis_perturbation} with respect to the following permutation already used in \cite{Fac14}. Let $b_0, b_1, b_2, \ldots$ be the first even numbers in the blocks $B_k \coloneqq [ \frac{(k-1)k}{2} + 1, \frac{k(k+1)}{2} ]$ $(k \in \IN)$. Now, define
			\[ \pi(m) = \begin{cases}
					m & m \text{ odd} \\
					b_k & m = 4k + 2 \\
					\min 2\IN \setminus (\{ b_n: n \in \IN \} \cup \pi ([1, m-1])) & m = 4k.
				 \end{cases}
			\]
		We need the following technical result proved in \cite[Proposition~6.6]{Fac14}.
		
		\begin{proposition}\label{prop:constructed_basis_unconditional} The basis $(f_m)_{m \in \IN}$ is unconditional for $p \in (1, 2]$. 
		\end{proposition}
				
		We are interested in the case $p > 2$. We make frequent use of the following technical observation. For a sequence $(a_m)_{m \in \IN}$ let $(b_m) = (0, \ldots, 0, a_1, 0, \ldots, 0, a_2, \ldots)$ be a sequence built from $(a_m)_{m \in \IN}$ by inserting zeros. Denote by $\phi\colon \IN \to \IN$ the mapping which sends $k$ to the position of $a_k$ in the new sequence $(b_m)_{m \in \IN}$. Then the following hold (for a proof of the first implication see~\cite[Lemma~6.7]{Fac14}, the second can be proved analogously).
		
		\begin{lemma}\label{lem:technical_lemma} Let $p \in [1,\infty)$, $(a_m)_{m \in \IN}$ be a sequence, $(b_m)_{m \in \IN}$ and $\phi\colon \IN \to \IN$ be as above and suppose that 
			\[ M \coloneqq \sup_{k \in \IN} \phi(k+1) - \phi(k) < \infty. \]
			If $(a_m)_{m \in \IN} \in X_p$, then $(b_m)_{m \in \IN} \in X_p$ as well. Conversely, if $(b_m)_{m \in \IN} \in X_p$, then $(a_m)_{m \in \IN} \in X_p$.
		\end{lemma}
		
		Recall that in the proof of Theorem~\ref{thm:positive_counterexamples} the fundamental property of $(\gamma_m)_{m \in \IN}$ used (as clarified in Lemma~\ref{lem:maximize_resolvent_sequence}) was that the ratios
			\begin{equation}\label{eq:ratio}
				\frac{\gamma_m - \gamma_{m-1}}{\gamma_m + \gamma_{m-1}}
			\end{equation}
		are bounded from below. We now study more precisely the Schauder multipliers associated to various sequences $(\gamma_m)_{m \in \IN}$ for the basis $(f_m)_{m \in \IN}$.
		
		As a starting point we make the very elementary observation that one can find sequences $(\gamma_m)_{m \in \IN}$ for which the ratio \eqref{eq:ratio} has a prescribed growth.
		
		\begin{lemma}\label{lem:existence_c_m}
			Let $(c_m)_{m \ge 2}$ be a sequence of real numbers with $c_m \in (0,\frac{1}{2})$ for all $m \in \IN$. Then there exists a unique strictly increasing sequence $(\gamma_m)_{m \in \IN}$ of real numbers with $\gamma_1 = 1$ and
				\begin{equation}\label{eq:define_c_m}
					\frac{1}{2} \frac{\gamma_m - \gamma_{m-1}}{\gamma_m + \gamma_{m-1}} = c_m \qquad \text{for all } m \ge 2.
				\end{equation}
		\end{lemma}
				
		We now formulate a necessary condition for the sequence $(c_m)_{m \in \IN}$ that implies that the Schauder multiplier associated to the sequence $(\gamma_m)_{m \in \IN}$ given by \eqref{eq:define_c_m} with respect to the basis $(f_m)_{m \in \IN}$ is $\mathcal{R}$-sectorial.
		
		\begin{proposition}\label{prop:mr_necessary_condition}
			Let $(c_m)_{m \ge 2}$ be a sequence with $c_m \in (0, \frac{1}{2})$ for all $m \ge 2$ and $(\gamma_m)_{m \in \IN}$ the sequence given by Lemma~\ref{lem:existence_c_m}. Suppose that for some $p > 2$ the sectorial operator $A$ on $X_p$ given as the Schauder multiplier
				\begin{align*}
					D(A) = \biggl\{ x = \sum_{m=1}^{\infty} a_m f_m: \sum_{m=1}^{\infty} \gamma_m a_m f_m \quad \mathrm{ exists} \biggr\} \\
					A \biggl( \sum_{m=1}^{\infty} a_m f_m \biggr) = \sum_{m=1}^{\infty} \gamma_m a_m f_m
			\end{align*}
			is $\mathcal{R}$-sectorial. Then $(a_m c_{4m+2})_{m \in \IN} \in X_p$ for all $(a_m)_{m \in \IN} \in \ell_p$.
		\end{proposition}
		\begin{proof}
			Observe that the basic sequence $(e_{\pi(4m+2)})_{m \in \IN}$ is isometrically equivalent to the standard unit vector basis of $\ell_p$. Let $(a_m)_{m \in \IN} \in \ell_p$. Then the Rademacher series $x = \sum_{m=1}^{\infty} r_m a_m e_{\pi(4m+2)}$ lies in $\Rad(X_p)$. One can now argue as in the proof of Theorem~\ref{thm:positive_counterexamples}:
			
			Let $(q_m)_{m \in \IN} \subset \IR_{-}$ be a sequence to be chosen later. Since $A$ is $\mathcal{R}$-sectorial by assumption, it follows from Proposition~\ref{prop:resolvent_r_associated_operator} that the operator $\mathcal{R}\colon \Rad(X) \to \Rad(X)$ associated to the sequence $(q_m)_{m \in \IN}$ is bounded. We now apply $\mathcal{R}$ to $x$. Because of $e_{\pi(4m+2)} = f_{4m+2} - f_{4m+1}$ we obtain
			\begin{align*}
				\mathcal{R}(x) & = \mathcal{R} \biggl( \sum_{m=1}^{\infty} r_m a_m (f_{4m+2} - f_{4m+1}) \biggr) \\
				& = \sum_{m=1}^{\infty} r_m \frac{a_m q_m}{q_m - \gamma_{4m+2}} f_{4m+2} - r_m \frac{a_m q_m}{q_m - \gamma_{4m+1}} f_{4m+1} \\
				& = \sum_{m=1}^{\infty} r_m \frac{a_m q_m}{q_m - \gamma_{4m+2}} (e_{\pi(4m+2)} + e_{4m+1}) - r_m \frac{a_m q_m}{q_m - \gamma_{4m+1}} e_{4m+1} \\
				& = \sum_{m=1}^{\infty} r_m \frac{a_m q_m}{q_m - \gamma_{4m+2}} e_{\pi(4m+2)} \\
				& \quad + r_m a_m q_m \biggl( \frac{1}{q_m - \gamma_{4m+2}} - \frac{1}{q_m - \gamma_{4m+1}} \biggr) e_{4m+1}.
			\end{align*}
			By Lemma~\ref{lem:maximize_resolvent_sequence} one has for $t = \gamma_{4m+2}$
				\[ 
					t [ (t + \gamma_{4m+2})^{-1} - (t + \gamma_{4m+1})^{-1} ] = -\frac{1}{2} \frac{\gamma_{4m+2} - \gamma_{4m+1}}{\gamma_{4m+2} + \gamma_{4m+1}} = -c_{4m+2}.
				\]
				Hence, for the choice $q_m = -\gamma_{4m+2}$ we obtain
				\[
					\mathcal{R}(x) = \sum_{m=1}^{\infty} \frac{1}{2} r_m a_m e_{\pi(4m+2)} - c_{4m+2} r_m a_m e_{4m+1}.
				\]
			As in the proof of Theorem~\ref{thm:positive_counterexamples} one deduces that $\sum_{m=1}^{\infty} c_{4m+2} a_m e_{4m+1}$ converges in $X_p$. By Lemma~\ref{lem:technical_lemma} this implies that $(a_m c_{4m+2})_{m \in \IN} \in X_p$. 
		\end{proof}
		
		In the next step we prove a sufficient criterion for maximal regularity. In fact, we will establish the boundedness of the imaginary powers. Let $A$ denote the Schauder multiplier associated to some sequence $(\gamma_m)_{m \in \IN}$ as above. It then follows from formula \eqref{eq:application_semigroup} that the imaginary powers $A^{it}$ for $t \in \IR$ act formally as
		\begin{equation}\label{eq:ait}
			\sum_{m=1}^{\infty} a_m e_m \mapsto \sum_{m=1}^{\infty} \tilde{\gamma}_m^{it} a_m e_m + \sum_{m=1}^{\infty} a_{2m} (\gamma_{\pi^{-1}(2m)}^{it} - \gamma_{\pi^{-1}(2m) - 1}^{it}) e_{\pi^{-1}(2m) - 1},
		\end{equation}
	where
		\[
			\tilde{\gamma}_m = \begin{cases}
				\gamma_m, & m \text{ odd} \\
				\gamma_{\pi^{-1}(m)}, & m \text{ even}
			\end{cases}.
		\]
	It is clear that the first series of the right hand side of \eqref{eq:ait} converges for all $(a_m)_{m \in \IN} \in X_p$. The crucial point is therefore the question whether the second series, which by the unconditionality of the basis $(e_m)_{m \in \IN}$ can be rewritten as
		\[
			\sum_{m=1}^{\infty} a_{2m} (\gamma_{\pi^{-1}(2m)}^{it} - \gamma_{\pi^{-1}(2m) - 1}^{it}) e_{\pi^{-1}(2m) - 1} = \sum_{m=1}^{\infty} a_{\pi(2m)} (\gamma_{2m}^{it} - \gamma_{2m-1}^{it})  e_{2m-1},
		\]
	converges in $X_p$ for all $(a_m)_{m \in \IN} \in X_p$. Equivalently by Lemma~\ref{lem:technical_lemma}, the sequence $(a_{\pi(2m)} (\gamma_{2m}^{it} - \gamma_{2m-1}^{it}))_{m \in \IN}$ must lie in $X_p$ for all $(a_m)_{m \in \IN} \in X_p$. We now give a sufficient condition. Here we use the fact that if a sectorial operator $A$ on some UMD-space has bounded imaginary powers of polynomial growth, then $A$ is $\mathcal{R}$-sectorial with $\omega_{R}(A) = 0$ \cite[Theorem~4.5]{DHP03}.	
		
	\begin{proposition}\label{prop:mr_sufficient_condition}
		Let $(c_m)_{m \in \IN}$ be a sequence with $c_m \in (0, \frac{1}{8})$ for all $m \ge 2$ and let $(\gamma_m)_{m \in \IN}$ be the sequence given by Lemma~\ref{lem:existence_c_m}. Consider for $p > 2$ the sectorial operator $A$ on $X_p$ defined as
		\begin{align*}
				D(A) = \biggl\{ x = \sum_{m=1}^{\infty} a_m f_m: \sum_{m=1}^{\infty} \gamma_m a_m f_m \quad \mathrm{ exists} \biggr\} \\
				A \biggl( \sum_{m=1}^{\infty} a_m f_m \biggr) = \sum_{m=1}^{\infty} \gamma_m a_m f_m.
		\end{align*}
		If $(b_m c_{2m})_{m \in \IN}$ lies in $X_p$ for all $(b_m)_{m \in \IN} \in \ell_p$, then $A$ has bounded imaginary powers with linear growth. In particular, $A$ is $\mathcal{R}$-sectorial with $\omega_R(A) = 0$.
	\end{proposition}
	\begin{proof}
		A short calculation shows that one has for all $m \in \IN$
			\begin{align*}
				\normalabs{\gamma_{2m}^{it} - \gamma_{2m-1}^{it}}^{2} & = \abs{\exp(it \log \gamma_{2m}) - \exp(it \log \gamma_{2m-1})}^2 \\ 
				& = \normalabs{\exp(it \log \gamma_{2m})}^2 + \normalabs{\exp(it \log \gamma_{2m-1})}^2 \\
				& \qquad - 2 \Re \exp(it (\log(\gamma_{2m-1} - \log \gamma_{2m}))) \\
				& = 2(1 - \cos(t (\log \gamma_{2m-1} - \log \gamma_{2m}))).
			\end{align*}
		Here we have used the identity
			\[
				\abs{z-w}^2 = (z-w)(\overline{z} - \overline{w}) = \abs{z}^2 + \abs{w}^2 - (z \overline{w} + \overline{z} w) = \abs{z}^2 + \abs{w}^2 - 2 \Re z \overline{w}.
			\]
		Further, one has
			\begin{align*}
				\abs{\log \gamma_{2m-1} - \log \gamma_{2m}} & = \abs{\log \left( \frac{\gamma_{2m-1}}{\gamma_{2m}} \right)} = \abs{\log \left( 1 - \frac{\gamma_{2m} - \gamma_{2m-1}}{\gamma_{2m}} \right)} \\
				& \le \abs{ \log \left( 1- 2 \frac{\gamma_{2m} - \gamma_{2m-1}}{\gamma_{2m} + \gamma_{2m-1}} \right)} = \abs{\log ( 1 - 4c_{2m})}.
			\end{align*}
		It follows from elementary calculus that $1 - \cos x \le \frac{x^2}{2}$ for all $x \in \IR$. In particular, we obtain the estimate
			\begin{align*}
				2(1 - \cos(t (\log \gamma_{2m-1} - \log \gamma_{2m})) \le t^{2} \log^2(1-4c_{2m}).
			\end{align*}
		A further elementary estimate from calculus is that $\abs{\log(1-4x)} \le 8x$ holds for all $x \in [0,\frac{1}{8}]$. Therefore we see that for all $m \in \IN$ one has
			\begin{equation}\label{eq:img_powers_fundamental_inequality}
				\normalabs{\gamma_{2m}^{it} - \gamma_{2m-1}^{it}} \le 8 \abs{t} c_{2m}. 
			\end{equation}
		Now, let $(a_m)_{m \in \IN} \in X_p$. Since $p > 2$, we have the inclusion $X_p \hookrightarrow \ell_p$. Hence, $(a_{\pi(2m)})_{m \in \IN} \in \ell_p$. By assumption, the mapping $(b_m)_{m \in \IN} \mapsto (b_m c_{2m})_{m \in \IN}$ from $\ell_p$ into $X_p$ is well-defined and closed. Hence, by the closed graph theorem there exists a constant $C \ge 0$ such that $\norm{(c_{2m} b_m)}_{X_p} \le C \norm{(b_m)}_{\ell_p}$ for all $(b_m)_{m \in \IN}$ in $\ell_p$. Hence, we obtain that $(a_{\pi(2m)} c_{2m})_{m \in \IN} \in X_p$ with 
		\[
			\normalnorm{(a_{\pi(2m)} c_{2m})}_{X_p} \le C \normalnorm{(a_{\pi(2m)})}_{\ell_p} \le C \norm{(a_m)}_{X_p}.
		\]
		It is now a direct consequence of equation \eqref{eq:img_powers_fundamental_inequality} that $((\gamma_{2m}^{it} - \gamma_{2m-1}^{it}) a_{\pi(2m)}) \in X_p$ with
		\[
			\normalnorm{((\gamma_{2m}^{it} - \gamma_{2m-1}^{it}) a_{\pi(2m)})}_{X_p} \le 8 C \abs{t} \norm{(a_m)}_{X_p}. 
		\]
		Altogether this shows that $A$ has bounded imaginary powers with $\normalnorm{A^{it}} \le K (1 + \abs{t})$ for some constant $K > 0$.
	\end{proof}
	
	\begin{remark}
		The same conditions as in Proposition~\ref{prop:mr_sufficient_condition} even imply the boundedness of the $H^{\infty}$-calculus for $A$ (for the necessary background see~\cite{KunWei04} and \cite{DHP03}). For this let first $g \in H^{\infty}(H_{\omega})$, where $H_{\omega} \coloneqq \{ z \in \IC : \abs{\Im z} < \omega \}$ is the strip of height $\omega > 0$. Then for $z \in \IR$ it follows from Cauchy's integral formula that for all $\tilde{\omega} \in (0, \omega)$ and all $k \ge 1$
			\begin{align*}
				\frac{\abs{g^{k}(z)}}{k!} & = \abs{\frac{1}{2\pi i} \left( \int_{-\infty + i\tilde{\omega}}^{\infty + i\tilde{\omega}} - \int_{-\infty - i\tilde{\omega}}^{\infty - i\tilde{\omega}} \right) \frac{g(w)}{(w-z)^{k+1}} \d w} \\
				 & \le \frac{\norm{g}_{H_{\omega}}}{\pi} \int_{-\infty}^{\infty} \frac{1}{\abs{s + i\tilde{\omega} - z}^{k+1}} \d s = \frac{\norm{g}_{H_{\omega}}}{\pi} \int_{-\infty}^{\infty} \frac{1}{(s^2 + \tilde{\omega}^2)^{(k+1)/2}} \d s \\
				 & = \frac{\norm{g}_{H_{\omega}}}{\pi} \tilde{\omega}^{-k} \int_{-\infty}^{\infty} \frac{1}{(1+s^2)^{(k+1)/2}} \d s \le \tilde{\omega}^{-k} \norm{g}_{H_{\omega}},
			\end{align*}
		where $H_{\omega}$ is endowed with the supremum norm. Hence, we obtain for $z_0, z \in \IR$ with $\abs{z-z_0} < \tilde{\omega}$ that
			\begin{align*}
				\abs{g(z) - g(z_0)} \le \sum_{k=1}^{\infty} \omega^{-k} \norm{g}_{H_{\omega}} \abs{z-z_0}^k \le C_{\tilde{\omega}} \norm{g}_{H_{\omega}} \abs{z-z_0}
			\end{align*}
		for a universal constant $C_{\tilde{\omega}} > 0$. Using the notation from Proposition~\ref{prop:mr_sufficient_condition} we obtain for $f \in H^{\infty}(\Sigma_{\theta})$ and $\theta > 0$
			\begin{align*}
				\MoveEqLeft \abs{f(\gamma_{2m}) - f(\gamma_{2m-1})} = \abs{(f \circ \exp)(\log \gamma_{2m}) - (f \circ \exp)(\log \gamma_{2m-1})} \\
				& \le C_{\tilde{\theta}} \norm{f \circ \exp}_{H_{\theta}} \abs{\log \gamma_{2m} - \log \gamma_{2m-1}} \le 8 C_{\tilde{\theta}} \norm{f}_{\Sigma_{\theta}} c_{2m}
			\end{align*}
		using the estimates from the proof of Proposition~\ref{prop:mr_sufficient_condition} provided $8c_{2m} < \tilde{\theta}$ for some $\tilde{\theta} \in (0, \theta)$. By the above assumptions $(c_m)_{m \in \IN}$ is a zero sequence and therefore this condition is satisfied for sufficiently large $m$ and we can deduce the boundedness of the $H^{\infty}$-calculus for $A$ with $\omega_{H^{\infty}}(A) = 0$ as in the above proof.
	\end{remark}
	
	For a special type of sequences $(c_m)_{m \in \IN}$ one can use the above results to obtain a complete characterization of maximal regularity.
			
	\begin{corollary}\label{cor:characterization_mr_via_sequences}
		Let $(c_m)_{m \in \IN}$ be an eventually decreasing sequence with $c_m \in (0, \frac{1}{8})$ for all $m \ge 2$ and $(\gamma_m)_{m \in \IN}$ the sequence given by Lemma~\ref{lem:existence_c_m}. Consider for $p > 2$ the sectorial operator $A$ on $X_p$ defined by
			\begin{align*}
				D(A) = \biggl\{ x = \sum_{m=1}^{\infty} a_m f_m: \sum_{m=1}^{\infty} \gamma_m a_m f_m \quad \mathrm{ exists} \biggr\} \\
				A \biggl( \sum_{m=1}^{\infty} a_m f_m \biggr) = \sum_{m=1}^{\infty} \gamma_m a_m f_m.
		\end{align*}
		Then $A$ is $\mathcal{R}$-sectorial if and only if $(c_m)_{m \in \IN} \in (\oplus_{n=1}^{\infty} \ell_q^n)_{\ell_{\infty}}$, where $\frac{1}{2} = \frac{1}{p} + \frac{1}{q}$. Moreover, in this case one has $\omega_R(A) = 0$. 	
	\end{corollary}
	\begin{proof}
		Clearly, it suffices to show the corollary for decreasing sequences. As a first observation we show that both the conditions in Proposition~\ref{prop:mr_necessary_condition} and Proposition~\ref{prop:mr_sufficient_condition} are equivalent to: $(a_m c_m)_{m \in \IN} \in X_p$ for all $(a_m)_{m \in \IN} \in \ell_p$. We show only the non-trivial implication for the condition in Proposition~\ref{prop:mr_sufficient_condition}. Of course, for the condition in Proposition~\ref{prop:mr_necessary_condition} the proof is completely analogous. So assume that $(a_m c_{2m})_{m \in \IN} \in X_p$ for all $(a_m)_{m \in \IN} \in \ell_p$. Now let $(a_m)_{m \in \IN} \in \ell_p$. In order to show that $(a_m c_{m})_{m \in \IN} \in X_p$, by Lemma~\ref{lem:technical_lemma}, it suffices to show that $(a_{2m} c_{2m})_{m \in \IN}$ and $(a_{2m+1} c_{2m+1})_{m \in \IN}$ lie in $X_p$. For the first sequence this follows directly from the assumption and for the second this follows from the monotonicity of $(c_m)_{m \in \IN}$.		
		
		Hence, we have shown that $A$ is $\mathcal{R}$-sectorial if and only if $(a_m c_m)_{m \in \IN} \in X_p$ for all $(a_m)_{m \in \IN} \in \ell_p$. In this case, by the closed graph theorem, there exists a constant $M \ge 0$ such that
			\begin{equation}\label{eq:characterizing_inclusion}
				\norm{(a_m c_m)}_{X_p} \le M \norm{(a_m)}_{\ell_p}
			\end{equation}
		for all $(a_m)_{m \in \IN} \in \ell_p$. Now, it remains to show that this condition is equivalent to $(c_m)_{m \in \IN} \in (\oplus_{n=1}^{\infty} \ell_q^n)_{\ell_{\infty}}$. On the one hand it follows from Hölder's inequality that for $(c_m)_{m \in \IN} \in (\oplus_{n=1}^{\infty} \ell_q^n)_{\ell_{\infty}}$ one has
			\begin{align*}
				\norm{(a_m c_m)}_{X_p} \le \sup_{m \in \IN} \bigg( \sum_{k \in B_m} \abs{a_k}^q \bigg)^{1/q} \bigg( \sum_{m=1}^{\infty} \abs{a_m}^p \bigg)^{1/p},
			\end{align*} 
		which is \eqref{eq:characterizing_inclusion}. On the other hand it follows from \eqref{eq:characterizing_inclusion} that for all $n \in \IN$
			\[
				\sup_{\norm{(a_m)}_{\ell_p} \le 1} \bigg( \sum_{k \in B_n} \abs{a_k c_k}^2 \bigg)^{1/2} \le M.
			\]
		This implies that for all $n \in \IN$ one has
			\[
				\bigg( \sum_{k \in B_n} \abs{c_k}^q \bigg)^{1/q} \le M.
			\]
		In other words one has $(c_m)_{m \in \IN} \in (\oplus_{n=1}^{\infty} \ell_q^n)_{\ell_{\infty}}$. This finishes the proof.
	\end{proof}
	
	We now give two fundamental examples for $(c_m)_{m \in \IN}$. First, for $c_m = k^{-\alpha}$ for $m \in B_k$ and $\alpha \in (0, \frac{1}{2})$ one has $(c_m)_{m \in \IN} \in (\oplus_{n=1}^{\infty} \ell_q^n)_{\ell_{\infty}}$ if and only if $p \le \frac{2}{1-2\alpha}$. Second, $c_m = k^{-\alpha} \log k$ for $m \in B_k$ lies in $(\oplus_{n=1}^{\infty} \ell_q^n)_{\ell_{\infty}}$ if and only if $p < \frac{2}{1-2\alpha}$. These two families of sequences can now be used to obtain a complete answer to the maximal regularity extrapolation problem.
	
	\begin{theorem}\label{thm:extrapolation_mr_complete_counterexample}\index{maximal regularity!extrapolation problem!counterexample}
		Let $I \subset (1, \infty)$ be an arbitrary interval with $2 \in I$. Then there exists a family of consistent analytic $C_0$-semigroups $(T_p(z))_{z \in \Sigma_{\frac{\pi}{2}}}$ on $L_p(\IR)$ for $p \in (1, \infty)$ such that $(T_p(z))_{z \in \Sigma_{\frac{\pi}{2}}}$ has maximal regularity (resp.\ bounded imaginary powers / a bounded $H^{\infty}$-calculus) if and only if $p \in I$.
	\end{theorem}
	\begin{proof}
		Let $I$ be such an interval and let $p_0$ be the right end of $I$. We first construct a family $(T_p(z))_{z \in \Sigma_{\frac{\pi}{2}}}$ that has maximal regularity if and only if $p \in (1,2) \cup I$. For $p_0 = 2$ this has already been done in \cite[Corollary~6.8]{Fac14}. So we may assume $p_0 > 2$. Choose $c_m = k^{-\alpha}$ for $m \in B_k$ and $\alpha = \frac{p_0 - 2}{2p_0}$ if $p_0 \in I$ or $c_m = k^{-\alpha} \log k$ for $m \in B_k$ and $\alpha = \frac{p_0 - 2}{2p_0}$ if $p_0 \not\in I$ multiplied by appropriate scaling constants such that $c_m \in (0, \frac{1}{8})$ for all $m \ge 2$. Then it follows from Corollary~\ref{cor:characterization_mr_via_sequences} and the above calculations that the analytic semigroups on $X_p$ for $p \in (1, \infty)$ whose generators are the Schauder multipliers associated the sequence $(-\gamma_m)_{m \in \IN}$ given by Lemma~\ref{lem:existence_c_m} with respect to the basis $(f_m)_{m \in \IN}$ have maximal regularity for $p \in (2, \infty)$ if and only if $p \in I \cap (2, \infty)$. Moreover, it follows from Proposition~\ref{prop:constructed_basis_unconditional} and \cite[Theorem~2.5]{Fac14} that these semigroups have maximal regularity for $p \in (1,2]$. In order to obtain consistent semigroups $(T_p(z))_{z \in \Sigma_{\frac{\pi}{2}}}$ on $L_p$ which have maximal regularity if and only if $p \in (1,2) \cup I$, one needs to transfer the just constructed example consistently in $p \in (1, \infty)$ from the $X_p$- to the $L_p$-scale. In fact, this can be done as in the proof of \cite[Theorem~6.3]{Fac14}:
		
		From the Khintchine inequality one obtains consistent isomorphisms $X_p = (\oplus_{n=1}^{\infty} \ell_2^n)_{\ell_p} \xrightarrow{\sim} (\oplus_{n=1}^{\infty} \Rad_n)_{\ell_p}$, where $\Rad_n$ is the span of the first $n$ Rademacher functions in $L_p([0,1])$. Hence, $(\oplus_{n=1}^{\infty} \Rad_n)_{\ell_p}$ can be identified with a closed subspace of $L_p([0,\infty))$. Together with the projection given by the direct sum of the consistent Rademacher projections $L_p([0,1]) \to \Rad_n$ we are able to transport the counterexample consistently to $L_p([0,\infty))$.  		 
		
		Taking dual semigroups, it follows from the first part of the proof that there exist consistent analytic $C_0$-semigroups $(S_p(z))_{z \in \Sigma_{\frac{\pi}{2}}}$ on $L_p([0,\infty))$ for $p \in (1, \infty)$ such that $(S_p(z))_{z \in \Sigma_{\frac{\pi}{2}}}$ has maximal regularity if and only if $p \in (2,\infty) \cup I$. Taking the direct sum of $(T_p(z))_{z \in \Sigma_{\frac{\pi}{2}}}$ and $(S_p(z))_{z \in \Sigma_{\frac{\pi}{2}}}$ one obtains the desired family of semigroups. 
	\end{proof}
	
	\begin{remark}
		It is an open problem whether every generator of a bounded analytic $C_0$-semigroup on a uniformly convex UMD-space that is contractive on $\IR_{\ge 0}$ has maximal regularity or even a bounded $H^{\infty}$-calculus. 
		
		We now comment on what we know about the contractivity of the semigroups considered above. First, it is easy to see that on $X_2 = \ell_2$ the semigroup given by a sequence $(c_m)_{m \in \IN}$ is contractive for all allowed $(c_m)_{m \in \IN}$.
		Further, on $X_{\infty} \coloneqq (\oplus_{n=1}^{\infty} \ell_2^n)_{c_0}$ one can again use the standard basis $(e_m)_{m \in \IN}$ to define the Schauder basis $(f_m)_{m \in \IN}$ and the sectorial operators as for example formulated in Corollary~\ref{cor:characterization_mr_via_sequences}. Using the same notation as before, the operator $B = -A$ given by a sequence $(c_m)_{m \in \IN}$ is the generator of a bounded analytic $C_0$-semigroup on $X_{\infty}$. Thus, by the Lumer--Phillips theorem, $B$ generates a contractive semigroup on $X_{\infty}$ if and only if $B$ is dissipative, i.e.\ if for all $x \in D(B)$ there exists an $x^* \in J(x) \coloneqq \{ x^* \in X_{\infty}^*: \langle x, x^* \rangle = \norm{x}^2 = \norm{x^*}^2 \}$ such that $\Re \langle Bx, x^* \rangle \le 0$. We now study this condition.
		
		Observe that one has $Be_m = -\gamma_m e_m$ for odd $m$ and $Be_{\pi(m)} = -(\gamma_m - \gamma_{m-1})e_{m-1} - \gamma_m e_{\pi(m)}$ for even $m$. Then one has for $x \in D(B)$ and $x^* \in X_{\infty}^*$
			\begin{align*}
				\langle Bx, x^* \rangle & = -\sum_{m \text{ odd}} \gamma_m x_m x^*_m - \sum_{m \text{ even}} \gamma_{\pi^{-1}(m)} x_{m} x^*_m \\
				& - \sum_{m \text{ odd}} (\gamma_{m+1} - \gamma_m) x_{\pi(m+1)} x^*_{m}.
		 	\end{align*}
		Now, choose $k \in \IN$ and $x_m = \frac{\gamma_{m+1} - \gamma_m}{2 \gamma_m}$ for $m \in B_k$ and $m \equiv 1 \mod 4$ and $x_{\pi(m+1)} = -1$ for $m \in B_k$ and $m+1 \equiv 2 \mod 4$ and $x_m = 0$ otherwise. Notice that if $k$ is sufficiently large and $\sum_{m \in B_k} \abs{x_m}^2 > 1$, then $x^* = \sum_{m \in B_k} x_m e_m$ is the unique element in $J(x)$. One therefore obtains
			\begin{align*}
				\langle Bx, x^* \rangle & = -\sum_{\substack{m \in B_k: \\m \equiv 1 \mod 4}} \gamma_m \abs{x_m}^2 + \sum_{\substack{m \in B_k: \\m \equiv 1 \mod 4}} (\gamma_{m+1} - \gamma_m) \abs{x_m} \\
				& = \sum_{\substack{m \in B_k: \\m \equiv 1 \mod 4}} \frac{1}{4} \frac{(\gamma_{m+1} - \gamma_m)^2}{\gamma_m} > 0.
			\end{align*}
		If $(c_m)_{m \in \IN} \not\in (\oplus_{n=1}^{\infty} \ell_2^n)_{\ell_{\infty}}$, one has $\sum_{m \in B_k} \abs{x_m}^2 > 1$ for sufficiently large $k$ because of the monotonicity of $(c_m)_{m \in \IN}$ and the estimate $c_m \le \frac{\gamma_{m+1} - \gamma_m}{2 \gamma_m} \le K c_m$ for some $K > 0$ and all $m \in \IN$.	Hence, it follows from the above calculation that $(c_m)_{m \in \IN} \not\in (\oplus_{n=1}^{\infty} \ell_2^n)_{\ell_{\infty}}$ implies the non-contractivity of the semigroup. Notice that $(c_m)_{m \in \IN} \not\in (\oplus_{n=1}^{\infty} \ell_2^n)_{\ell_{\infty}}$ is exactly the limit case for $p \to \infty$ of the condition given in Corollary~\ref{cor:characterization_mr_via_sequences}.
		
		We do not know whether analogously for $p \in (2, \infty)$ the condition $(c_m)_{m \in \IN} \not\in (\oplus_{n=1}^{\infty} \ell_q^n)_{\ell_{\infty}}$ with $\frac{1}{2} = \frac{1}{p} + \frac{1}{q}$ implies that the generated semigroup is not contractive.
	\end{remark}
	
	\printbibliography  

\end{document}